\documentclass{amsart}
\usepackage[utf8]{inputenc}

\newtheorem{theorem}{Theorem}[section]
\newtheorem{lemma}[theorem]{Lemma}
\newtheorem{proposition}[theorem]{Proposition}
\newtheorem{corollary}[theorem]{Corollary}

\theoremstyle{definition}
\newtheorem{remark}[theorem]{Remark}
\newtheorem{definition}[theorem]{Definition}
\newtheorem{example}[theorem]{Example}

\newcommand{\cL}{\mathcal{L}}

\renewcommand{\S}{\mathcal{S}}

\newcommand{\bF}{\mathbb{F}}

\newcommand{\<}{\langle}
\renewcommand{\>}{\rangle}

\DeclareMathOperator{\alg}{Alg}
\DeclareMathOperator{\spn}{span}

\newcommand{\sloc}{S\l oci\'{n}ski}

\title{S\l oci\'nski-Wold decompositions for row-isometries}
\author{Adam H. Fuller}
\address{Department of Mathematics, Ohio University, Athens, OH 45701}
\email{fullera@ohio.edu}
\date{\today}

\begin{document}

\begin{abstract}
\sloc\ gave sufficient conditions for commuting isometries to have a nice Wold-like decomposition.
In this note we provide analogous results for row-isometries satisfying certain commutation relations.
Other than known results for doubly-commuting row-isometries, we provide sufficient condtions for a Wold decomposition based on the Lebesgue decomposition of the row-isometries.
\end{abstract}

\maketitle

\section{Introduction}
Let $V$ be an isometry acting on a Hilbert space $H$.
A well-known result, discovered independently by von Neumann (1929) and Wold (1938), tells us that $H$ decomposes uniquely into $V$-reducing subspaces $H=H_u\oplus H_s$ where $V|_{H_u}$ is a unitary and $V_{H_s}$ is a unilateral shift. We will follow the convention of calling this result the \emph{Wold decomposition of $V$}.
Over the decades there have been generalisations of this result, decomposing isometric representations of semigroups into their unitary and non-unitary parts.
Suciu's work in \cite{Suc1968} is an early example of such results. 

The work at hand is largely inspired by the Wold-like decomposition given \sloc\ \cite{Slo1980}.
Let $V_1$ and $V_2$ be commuting isometries on a Hilbert space $H$.
We say $V_1$ and $V_2$ have a \emph{\sloc-Wold} decomposition if $H$ decomposes as $H = H_1 \oplus H_2 \oplus H_3 \oplus H_4$, where each space $H_i$ reduces both $V_1$ and $V_2$;
$V_1|_{H_1},V_1|_{H_2},V_2|_{H_1},V_2|_{H_3}$ are unitaries; and
$V_1|_{H_3},V_1|_{H_4},V_2|_{H_2},V_2|_{H_4}$ are unilateral shifts.
\sloc\ gives sufficient conditions for a pair commuting isometries to have a \sloc-Wold decomposition.
Most notable, or at least the the most noted, of these results is that a pair of \emph{doubly}-commuting isometries $V_1$ and $V_2$ has a \sloc-Wold decomposition (where doubly commuting means that $V_1V_2=V_2V_1$ and $V_1^*V_2 = V_2V_1^*$).
Generalisations of this result for $n$ doubly-commuting isometries have been given \cite{GasSuc1989}.
\sloc\ also gives sufficient conditions for the existence of a \sloc-Wold decomposition based on the structure of the individual unitary parts of the isometries.
Recall that a unitary $U$ can decomposed as $U_{abs} \oplus U_{sing}$ where $U_{abs}$ has absolutely continuous spectral measure and $U_{sing}$ has singular spectral measure (both with respect to Lebesgue measure).
\sloc\ gives two results \cite[Theorem~4 and Theorem~5]{Slo1980} showing the existence of a \sloc-Wold decomposition in the absence of of absolutely continuous unitary parts.

Let $S = [S_1,\ldots,S_m]$ be a row-isometry on a Hilbert space $H$.
That is $S \colon H^{(m)} \rightarrow H$ is an isometric map.
Equivalently, $S = [S_1,\ldots,S_m]$ is a row-isometry if $S_1, \ldots, S_n$ are isometries with pairwise orthogonal ranges.
Popescu \cite{Pop1989} shows there is a Wold-decomposition for $S$.
That is $H$ can be decomposed into $S$-reducing subspaces $H = H_u \oplus H_s$  where $S|_{H_u}$ is a row-unitary, and $S|_{H_s}$ is an $n$-shift.
Beyond row-isometries, Muhly and Solel \cite{MuhSol1999} give a Wold decomposition isometric representations of C$^*$-correspondences, decomposing an isometric representation into unitary and induced parts.

Let $S=[S_1,\ldots,S_m]$ and $T=[T_1,\ldots, T_n]$ be two row-isometries on a Hilbert space $H$.
We say $S$ and $T$ $\theta$-commute if there is a permutation $\theta \in S_{m\times n}$ such that for $1\leq i \leq m$ and $1\leq i \leq n$, $S_iT_j = T_{j'}S_{i'}$ when $\theta(i,j) = (i',j')$.
A pair of $\theta$-commuting row-isometries determines an isometric representation of a $2$-graph with a single vertex.
Thus, a pair of $\theta$-commuting row-isometries is an isometric representation of a product system of two finite-dimensional C$^*$-correspondences, see e.g. \cite[Section~4]{Ful2011}.
Skalski and Zacharias \cite{SkaZac2008} generalised of \sloc's Wold decomposition for doubly-commuting isometries to isometric representations of product systems of C$^*$-correspondences which satisfy a doubly-commuting condition.
Thus, Skalski and Zacharias's result gives a \sloc-Wold decomposition for $\theta$-commuting row-isometries.

In this note we will give sufficient conditions for two $\theta$-commuting row-isometries to have a \sloc-Wold decomposition mirroring the three theorems proved by \sloc\ for commuting isometries.
Theorems 3, 4 and 5 of \cite{Slo1980} are generalized in Theorems \ref{thm: skalski zach}, \ref{thm: singular parts} and \ref{thm: singular parts 2} respectively.
In \cite[Theorem~4 and Theorem~5]{Slo1980}, \sloc\ uses the Lebesgue decomposition of a unitary.
For row-unitaries we use the Lebesgue decomposition due to Kennedy \cite{Ken2013}.
This states that any row-unitary decomposes into an absolutely continuous row-unitary, a singular row-unitary, and a third part called a dilation-type row-unitary.
For a single unitary $U$ the statements `$U$ has no absolutely continuous part' and `$U$ is singular' are equivalent; for row-unitaries the existence of dilation-type parts means that the latter is a stronger statement than the former.
In this note, for a row-unitary, the statement `$U$ is singular' will play the role that `$U$ has no absolutely continuous part' played in \cite{Slo1980}.

\section{Row-isometries and their structure}

A \emph{row-isometry} on a Hilbert space $H$ is an isometric map $S$ from $H^{(n)}$ to $H$.
An operator $S \colon H^{(n)} \rightarrow S$ is a row-isometry if and only if $S = [S_1, \ldots, S_m]$ where $S_1, \ldots, S_m$ are isometries on $H$ with pairwise orthogonal ranges.
Equivalently, the $S_1, \ldots, S_m$ are isometries satisfying
$$ \sum_{i=1}^m S_iS_i^* \leq I_H. $$
A row-isometry $S = [S_1, \ldots, S_m]$ is a \emph{row-unitary} if $S$ is a unitary map.
Equivalently, $S$ is a row-unitary if
$$ \sum_{i=1}^m S_iS_i^* = I_H. $$

Let $S=[S_1, \ldots, S_m]$ be a row-operator on a Hilbert space $H$ and let $M \subseteq H$ be a subspace.
The subspace $M$ is \emph{$S$-invariant} if $S_i H \subseteq H$ for each $1\leq i \leq m$;  $M$ is \emph{$S^*$-invariant} if $S_i^* H \subseteq H$ for each $1\leq i \leq m$; and $M$ is \emph{$S$-reducing} if $M$ is both $S$-invariant and $S^*$-invariant.

Denote by $\bF_m^+$ the unital free-semigroup on $n$ generators $\{1,\ldots, m\}$.
For $w = w_1 \ldots w_k \in \bF_n^+$ denote by $S_w$ the isometry
$$ S_{w_1}S_{w_2}\ldots S_{w_k}. $$
Here $S_\emptyset$ will denote $I_H$.

\begin{example}\label{ex: L}
Let $H= \ell^2(\bF_m^+)$ with orthonormal basis $\{ \xi_w \colon w \in \bF_m^+ \}.$
For $i \in \{1, \ldots, m\}$, define the operator $L_i$ by
$$ L_i \xi_w = \xi_{iw}. $$
Then $L = [L_1, \ldots, L_m]$ is a row-isometry on $H$. 
\end{example}

\begin{definition}
Let $S = [S_1, \ldots, S_m]$ be a row-isometry.
Let $L$ be the row-isometry described in Example~\ref{ex: L},
We call $S$ and \emph{$m$-shift of multiplicity $\alpha$} if $S$ is unitarily equivalent to an ampliation of $L$ by $\alpha$.
That is $[S_1,\ldots,S_m] \simeq [L_1^{(\alpha)},\ldots,L_m^{(\alpha)}]$.
\end{definition}

Note that when $m=1$, an $m$-shift is a unilteral shift.
Thus, the following result, due to Popescu \cite{Pop1989}, is a generalisation of the Wold decomposition of a single isometry.

\begin{theorem}[{cf.~\cite[Theorem~1.2]{Pop1989}}]\label{thm: wold}
Let $S = [S_1, \ldots, S_m]$ be a row-isometry on $H$.
Then $H$ decomposes into two $S$-reducing subspaces
$$H = H_u \oplus H_{s}, $$
such that $S|_{H_u}$ is a row-unitary and $S|_{H_{s}}$ is an $m$-shift.

Further
$$ H_u = \bigcap_{k\geq 0} \bigoplus_{|w| = k} S_w H, $$
and
$$ H_{s} = \bigoplus_{w\in \bF_n^+}S_w M $$
where $M = \bigcap_{i=1}^n \ker(S_i^*)$.
\end{theorem}

\begin{definition}
When $S$ is a row-isometry on a Hilbert space $H$, the decomposition $H = H_s \oplus H_u$ described in Theorem~\ref{thm: wold} is called the \emph{Wold decomposition} of $S$.
\end{definition}

\subsection{The Lebesgue-Wold decomposition}
Just as a unitary can be decomposed into its singular and unitary parts, a row-unitary can be decomposed further. 
We will briefly summarise these results now, drawing largely from \cite{DKP2001} and \cite{Ken2013}.

Let $L = [L_1, \ldots, L_m]$ be the $m$-shift described in Example~\ref{ex: L}.
Denote by $A_m$ and $\cL_m$ the following two algebras:
\begin{align*}
    A_m &:= \alg\{I, L_1, \ldots, L_m \overline{\}}^{\|\cdot\|} \\
    \cL_m &:= \alg\{I, L_1, \ldots, L_m \overline{\}}^{\textsc{wot}}
\end{align*}
The algebra $A_m$ is called the \emph{noncommutative disk algebra}; and the algebra $\cL_m$ is called the \emph{noncommutative analytic Toeplitz algebra}.

Let $S = [S_1, \ldots, S_m]$ be a row-isometry on a Hilbert space $H$. 
The \emph{free semigroup algebra} generated by $S$ is the algebra
$$ \S := \alg\{I, S_1, \ldots, S_m \overline{\}}^{\textsc{wot}}. $$
Popescu \cite{Pop1996} observed that the unital, norm-closed algebra generated by $S_1, \ldots, S_m$ is completely isometrically isomorphic to the noncommutative disk algebra $A_m$.
The free semigroup algebra $\S$, however, can be very different from $\cL_m$.

\begin{definition}\label{def: abs cont}
Let $S = [S_1,\ldots,S_m]$ be a row-isometry on a Hilbert space $H$ with $m\geq 2$.
\begin{enumerate}
\item There is a completely isometric isomorphism
$$ \Phi \colon A_m \rightarrow \alg\{I, S_1, \ldots, S_m\overline{\}}^{\|\cdot\|}, $$
such that $\Phi(L_i) = S_i$ for $1\leq i \leq m$.
The row-isometry $S$ is \emph{absolutely continuous} if $\Phi$ extends to a weak-$^*$ continuous representation of $\cL_m$.

\item The row-isometry $S$ is \emph{singular} if $S$ has no absolutely continuous restriction to an invariant subspace.

\item The row-isometry $S$ is of \emph{dilation-type} if it has no singular and no absolutely continuous summands.
\end{enumerate}
\end{definition}

\begin{remark}\label{rem: unitary types}
\begin{enumerate}
\item Absolute continuity for row-isometries was introduced by Davidson, Li and Pitts \cite{DLP2005}.
We refer the reader to \cite[Section~2]{DLP2005} or \cite[Section~2]{Ken2013} for details on why Definition~\ref{def: abs cont} generalizes the notion of a unitary with absolutely continuous spectral measure.

\item By \cite[Theorem~5.1]{Ken2013}, a row-isometry $S = [S_1,\ldots, S_m]$, with $m\geq 2$, is singular if and only if the free semigroup algebra $\S$ generated by $S$ is a von Neumann algebra.
Read \cite{Rea2005} gave the first example of a selfadjoint free semigroup algebra, by showing that $B(H)$ is a free semigroup algebra, see also \cite{Dav2006}.

\item The name `dilation-type' is justified in \cite[Proposition~6.2]{Ken2013}. If $S$ is a row-isomtry of dilation-type on $H$, then there is a minimal subspace $V \subseteq H$ such that $V$ is invariant for each $S_i^*$, $1\leq i \leq m$ and the restriction of $S$ to $V^\perp$ is an $m$-shift.
In which case, $S$ is the minimal isometric dilation of the compression of $S$ to $V$.
\end{enumerate} 
\end{remark}

We can now describe the Lebesgue-Wold decomposition of a row-isometry, due to Kennedy \cite{Ken2013}.

\begin{theorem}[{cf.~\cite[Theorem~6.5]{Ken2013}}]\label{thm: lebesgue wold}
If $S$ is a row-isometry on $H$ then $H$ decomposes into $4$ spaces which reduce $S$.
$$ H = H_{abs} \oplus H_{sing} \oplus H_{dil} \oplus H_{s},$$
where $H_{abs} \oplus H_{sing} \oplus H_{dil}$ and $H_s$ are the unitary and $m$-shift parts of the Wold decomposition respectively.
Further we have the following properties:
\begin{enumerate}
    \item $S|_{H_{abs}}$ is absolutely continuous;
    \item $S|_{H_{sing}}$ is singular;
    \item $S|_{H_{dil}}$ is of dilation-type.
\end{enumerate}
\end{theorem}

Kennedy \cite[Theorem~4.16]{Ken2013} gives another characterisation of absolute continuity.
Let $S=[S_1,\ldots, S_m]$ be a row-isometry with $m\geq 2$, and let $\S$ be the free semigroup algebra generated by $S$.
Then $S$ is absolutely continuous if and only if $\S$ is isomorphic to $\cL_m$.
This characterisation answered a question asked in \cite{DLP2005}.

The property of $\S$ being isomorphic to $\cL_m$ plays an important role in the work of Davidson, Katsoulis and Pitts \cite{DKP2001} in describing the structure of free semigroup algebras.
We summarise the results which will be relevant to us now.
Note that what we are calling `absolutely continuous' was called `\emph{type L}' in \cite{DKP2001}.
The equivalence of the terms is due to the aforementioned work of Kennedy \cite{Ken2013}.

\begin{theorem}[{cf.~\cite[Theorem~2.6]{DKP2001}}]\label{thm: structure thm}
Let $S = [S_1,\ldots,S_m]$ be a row-isometry on a Hilbert space $H$ with $m\geq 2$.
Let $\S$ be the free semigroup algebra generated by $\S$.
Then is a largest projection $P$ in $\S$ such that $P\S P$ is self-adjoint.
Further, the following are satisfied:
\begin{enumerate}
    \item $PH$ is $S^*$-invariant; and
    \item the restriction of $S$ to $P^\perp H$ is an absolutely continuous row-isometry.
\end{enumerate}
\end{theorem}

\begin{definition}
Let $S$ be a row-isometry and let $P$ be the projection described in Theorem~\ref{thm: structure thm}.
Then $P$ is called \emph{structure projection} for $S$.
\end{definition}

Let $S = [S_1, \ldots, S_m]$ be a row-isometry on $H$, with $H = H_{abs} \oplus H_{sing} \oplus H_{dil} \oplus H_{s}$ being the Lebesgue-Wold decomposition.
Further write $H_{dil} = V \oplus \bigoplus_{w\in \bF_n^+} S_w K$, as described in Remark~\ref{rem: unitary types}(iii).
It follows from Theorem~\ref{thm: lebesgue wold} and Theorem~\ref{thm: structure thm} that
$$ PH = H_{sing} \oplus V. $$
If $S$ is a row-isometry of dilation-type on $H = V \oplus \bigoplus_{w\in\bF_n^+} S_w K$, then for each $v \in V$ there is a $w \in \bF_n^+$ so that $S_w v \notin V$, see e.g. \cite[Lemma~4.9]{FulYan2015}.
This is precisely the property which separates $V$ from $H_{sing}$ in $PH$.

\begin{corollary}\label{cor: H sing}
If $h \in PH$, then $h \in H_{sing}$ if and only if $S_w h \in PH$ for all $w\in \bF_n^+$
\end{corollary}

\begin{proof}
If $h \in H_{sing}$ then $S_w h \in H_{sing}$ for all $w\in \bF_n^+$, since $H_{sing}$ reduces $S$.

Conversely, let $M = \spn\{S_w h \colon w \in \bF_n^+\overline{\}}^{\|\cdot\|}$ for some $h \in H_{ing}$.
Assume $Pm = m$ for all $m \in M$.
As $M$ is $\S$ invariant, and $M \subseteq PH$, it follows that $M$ is $\S^*$-invariant.
Thus $M$ reduces $S$ and $S|_M$ is singular. 
Hence $M \subset H_{sing}$.
\end{proof}

%%%%%%%%%%%%%%%%%%%%%%%%%%%%%%%%%%%%%%%%%%%%
\section{S\l oci\'nski-Wold decompositions for $\theta$-commuting row-isometries}
\begin{definition}
Let $A=[A_1,\ldots,A_m]$ and $B=[B_1,\ldots,B_n]$ be two row-operators on a Hilbert space $H$, and let $\theta \in S_{m\times n}$ be a permutation.
We say that $A$ and $B$ \emph{$\theta$-commute} if
$$ A_i B_j =B_{j'}A_{i'} $$
when $\theta(i,j) = (i',j')$.
When $\theta$ is the identity permutation we will say that $A$ and $B$ \emph{commute}.

If $A$ and $B$ are $\theta$-commuting row-operators which further satisfy
\begin{align*}
 B_j^* A_i &= \sum_{\theta(k,j) = (i,j_k)} A_k B_{j_k}^*, \text{ and}\\
 A_i^* B_j &= \sum_{\theta(i,k) = (i_k, j)}B_k A_{i_k}^*
\end{align*}
we say that $A$ and $B$ \emph{$\theta$-doubly commute}.
\end{definition}

The following lemma is proved by repeated applications of the commutation rule from $\theta$.
It will be used liberally in the sequel.

\begin{lemma}
Let $A=[A_1,\ldots,A_m]$ and $B=[B_1,\ldots,B_n]$ be $\theta$-commuting row-operators.
For each $k, l \geq 1$ $\theta$ determines a permutation $\theta_{k,l} \in S_{m^k \times n^l}$ so that
$$ A_u B_w = B_{w'} A_{u'} $$
when $\theta_{k,l}(u,w) = (u',w')$.
\end{lemma}

Any $2$-graph with a single vertex, in the sense of \cite{KumPas2000}, is uniquely determined by a single permutation.
Thus, two contractive $\theta$-commuting row-contractions $A$ and $B$ determine a contractive representation of single vertex $2$-graph.
This is the perspective $\theta$-commuting row-operators are studied from in, e.g., \cite{DPY2010, DavYan2009, FulYan2015}.

\begin{definition}
Let $S=[S_1,\ldots, S_m]$ and $T=[T_1,\ldots,T_n]$ be $\theta$-commuting row-isometries on a Hilbert space $H$.
We say that $S$ and $T$ have a \emph{\sloc-Wold decomposition} if $H$ decomposes into
$$ H = H_{uu} \oplus H_{us} \oplus H_{su} \oplus H_{ss}, $$
where $H_{uu}$, $H_{us}$, $H_{su}$ and $H_{ss}$ are both $S$-reducing and $T$-reducing subspaces satistfying:
\begin{enumerate}
    \item $S|_{H_{uu}}$ and $T|_{H_{uu}}$ are both row-unitaries;
    \item $S|_{H_{us}}$ is a row-unitary and $T|_{H_{us}}$ is an $n$-shift;
    \item $S|_{H_{su}}$ is an $m$-shift and $T|_{H_{su}}$ is a row-unitary;
    \item $S|_{H_{ss}}$ is an $m$-shift and $T|_{H_{ss}}$ is an $n$-shift.
\end{enumerate}
\end{definition}

The following general lemma will be used throughout our analysis.

\begin{lemma}\label{lem: Hu inv}
$S=[S_1,\ldots,S_m]$ a row-isometry which $\theta$-commutes with a row-operator $A = [A_1,\ldots, A_l]$.
Let $H = H_u \oplus H_s$ be the Wold decomposition of $S$.
Then $H_u$ is $A$-invariant.
\end{lemma}

\begin{proof}
Take $h \in H_u$ and fix $k \geq 0$.
Since $S$ is a row-unitary on $H_u$,
$$ h = \sum_{|w|=k} S_w S_w^* h.$$
Choose an $A_i$, $1\leq i \leq l$.
For each $w$ with $|w|=k$ there is a $w'$ with $|w'|=k$, and $i_w$ with $1\leq i_w \leq l$ so that $A_iS_w = S_{w'}A_{i_w}$.
Thus
\begin{align*}
    A_i h & = A_i \sum_{|w|=k} S_w S_w^* h \\
    & = \sum_{|w|=k}S_{w'}A_{i_w}S_w^*h
     \in \sum_{|w|=k} S_w H.
\end{align*}
Since this holds for all $k\geq 0$, $A_i H_u \subseteq H_u$ by Theorem~\ref{thm: wold}.
\end{proof}

We can now give a general statement on the existence of \sloc-Wold decompositions.
The case when $m = n = 1$ is covered in \cite[Proposition~3]{Slo1980}.

\begin{proposition}\label{prop: sloc exists}
Let $S = [S_1, \ldots, S_m]$ and $T = [T_1,\ldots,T_n]$ be $\theta$-commuting row-isometries on $H$.
Then $S$ and $T$ have a \sloc-Wold decomposition if and only if
\begin{enumerate}
    \item if $H = H_u^S \oplus H_s^S$ is the Wold decomposition of $S$, then $H_u^S$ reduces $T$; and
    \item if $H_s^S = H_u^T \oplus H_s^T$ is the Wold decomposition of $T|_{H_s^S}$, then $H_u^T$ reduces $S$.
\end{enumerate}
\end{proposition}

\begin{proof}
If $S$ and $T$ have a \sloc-Wold decomposition then conditions (1) and (2) are clearly satisfied.

Suppose now that conditions (1) and (2) are satisfied.
Let $H = H_u^S \oplus H_s^S$ be the Wold decomposition for $S$.\
Let $H_u^S = K_u^T \oplus K_s^T$ be the Wold decomposition of $H_u^S$ from the restriction of $T$ to $H_u^S$.
By Lemma~\ref{lem: Hu inv}, $K_u^T$ is $S$-invariant.
Take any $1\leq i \leq m$, and $h \in K_u^T$.
Then for every $k\geq 1$
\begin{align*}
    S_i^*h & = S_i^* \sum_{|w|=k}T_wT_w^* h \\
    & =  \sum_{|w|=k}S_i^* T_wT_w^* h \\
    &= \sum_{|w|=k}\sum_{l = 1}^m S_i^* T_w S_l S_l^* T_w^* h \\
    &= \sum_{|w|=k}\sum_{\theta_{1,k}(i,w_i) = (l,w)}T_{w_i}S_l^* T_w^* h \\
    & \in \bigoplus_{|w|=k}T_w H_u^S,
\end{align*}
where the fact that $S$ is a row-unitary on $H_u^S$ is used in the third equality.
It follows from Theorem~\ref{thm: wold} that $S_i^*h \in K_u^T$.
Hence $K_u^T$ is $S$-reducing.

Letting $H_s^S = H_u^T \oplus H_s^T$ be the Wold decomposition of $T|_{H_u^S}$, we have that
$H_{uu} = K_u^T$, $H_{us} = K_s^T$, $H_{su} = H_u^T$, and $H_{ss} = H_s^T$ gives the desired \sloc-Wold decomposition.
\end{proof}

Skalski and Zacharias studied Wold decompositions of isometric representations of product systems of C$^*$-correspondences \cite{SkaZac2008}.
The following is a special case of one of their results.

\begin{theorem}[{cf. \cite[Theorem~2.4]{SkaZac2008}}]\label{thm: skalski zach}
If $S$ and $T$ are $\theta$-double commuting row-isometries then they have a \sloc-Wold decomposition.
\end{theorem}

\begin{proof}
Let $H = H^S_u \oplus H^S_s$ be the Wold decomposition of $H$ from $S$.
We we will show that $H_u^S$ is $T$-reducing.
Lemma~\ref{lem: Hu inv} gives that $H_u^S$ is $T$-invariant, so it only remains to show that $H_u^S$ is $T^*$-invariant.
Take $1\leq j \leq n$ and $h \in H_u^S$.
Using the condition that $S$ and $T$ $\theta$-doubly commute and that $S$ is a row-unitary on $H_u^S$ we have for every $k \geq K$
\begin{align*}
    T_j^* h &= \sum_{|w|=k} T_j^* S_w S_w^*h \\
    & = \sum_{\theta_{k,1}(w_k,j) = (w,j_w)} S_{w_k}T_{j_w}^* S_w^* h\\
    &\in \sum_{|w|=k}S_w H.
\end{align*}
Thus $T_j^* h \in H_u^S$ by Lemma~\ref{thm: wold}.

Now, let $H_s^S = H_u^T \oplus H_s^T$ be the Wold decomposition of $T|_{H_s^S}$.
The same calculation as above, we the roles of $S$ and $T$ swapped, shows that $H_u^T$ is $S$-reducing.
Thus $S$ and $T$ have a \sloc-Wold decomposition by Proposition~\ref{prop: sloc exists}.
\end{proof}

\begin{remark}
As described in \cite{SkaZac2008}, the \sloc-Wold decomposition for $\theta$-doubly commuting row-isometries has additional structure on the shift part $H_{ss}$.
On $H_{ss}$ $S$ and $T$ are not just both ($m$ and $n$) shifts.
The operators $S$ and $T$ work as shifts \emph{together}, giving an ampliation of the left-regular representation of the unital semigroup
$$ F_\theta^+ = \< \ i_1,\ldots,i_m, j_1,\ldots, j_n \colon i_k j_l = j'i' \text{ when } \theta(i_k,j_l)=(i',l') \>. $$
Explicitly, if $M = \bigcap_{i=1}^m \ker S_i^* \cap \bigcap_{j=1}^n \ker T_j^*$ then
$$ H_{ss} = \bigoplus_{u \in \bF_m^+,\ w\in \bF_n^+} S_u T_w M. $$
\end{remark}

Theorem~\ref{thm: skalski zach} generalises Theorem~3 of \cite{Slo1980}.
In the rest of this note we will give analogues of Theorem~4 and Theorem~5 of \cite{Slo1980} for $\theta$-commuting row-isometries.
That is, we will give sufficient conditions for the existence of a \sloc-Wold decomposition for $\theta$-commuting row-isometries based on the Lebesgue decomposition of their unitary parts.

Let $V$ be an isometry on a Hilbert space $H$ and $N \in B(H)$ be an operator commuting with $V$.
Let $H = H_{abs} \oplus H_{sing} \oplus H_s$ be the Lebesgue-Wold decomposition of $V$.
It then follows from \cite[Theorem~2.1]{Mla1972} that $H_{sing}$ reduces $N$.
This is a key tool in \cite[Theorem~4]{Slo1980} and \cite[Theorem~5]{Slo1980}.
We show now that this holds for operators commuting with row-isometries also.
Unfortunately, this result is not readily applicable to $\theta$-commuting row-operators.

\begin{proposition}\label{prop: Hsing red}
Let $S = [S_1,\ldots,S_m]$ be a row-isometry on $H$ and let $N$ be an operator which commutes with each $S_i$, $1\leq i \leq m$.
If the Lebesgue-Wold decomposition of $S$ decomposes $H$ as
$H_{abs} \oplus H_{sing} \oplus H_{dil} \oplus H_{s}$, then $H_{sing}$ reduces $N$.
\end{proposition}

\begin{proof}
Let $P$ be the structure projection for $S$.
Since $P$ lies in the free semigroup algebra generated by $S$, $N$ commutes with $P$.
Thus $PH$ reduces $N$.
Now take any $h \in H_{sing}$, then for any $w\in \bF_n^+$
$$ PS_w Nh = N P S_w h = N S_w h = S_w N h. $$
Thus $Nh \in H_{sing}$ by Corollary~\ref{cor: H sing}.

As $P$ is self-adjoint we also have that $N^*$ commutes with $P$.
Thus, for any $h \in H_{sing}$, $N^*h \in PH \subseteq H_u$, the unitary part of $S$.
Hence, for any $w\in \bF_n^+$ and $h \in H_{sing}$
$$ N^* S_w h  = \sum_{|u|=|w|} S_uS_u^* N^* S_wh = S_w N^* h. $$
Hence $S_wN^* h$ is in  $PH$.
Again, by Corollary~\ref{cor: H sing}, it follows that $N^* h \in H_{sing}$.
Thus $H_{sing}$ reduces $N$.
\end{proof}

For $\theta$-commuting operators it is not clear that a direct analogue of Proposition~\ref{prop: Hsing red} holds.
We do, however, have the following lemma for $\theta$-commuting row-isometries
Proposition~\ref{prop: Hu reduce} shows this will suffice for our purposes.

\begin{lemma}\label{lem: PH *-invariant}
Let $S=[S_1,\ldots,S_m]$ be a row-isometry on $H$ with $m\geq 2$, and let $P$ be the structure projection for $S$.
If $T = [T_1,\ldots, T_n]$ is a row-isometry on $H$ which $\theta$-commutes with $S$.
Then $PH$ is $T^*$-invariant.
\end{lemma}

\begin{proof}
By Theorem~\ref{thm: wold}, $S$ is absolutely continuous on $P^\perp H$.
Thus, by \cite[Corollary~4.17]{Ken2013}, $P^\perp H$ is spanned by wandering vectors for $S$.
Let $h$ be a wandering vector for $S$. 
Then for any $1\leq j \leq n$ and $w \in \bF_n^+$, $|w|\geq 1$ we have
\begin{align*}
    \< S_w T_j h, T_j h\> &= \<S_{w'}h, T_{j'}^*T_j h\>, 
\end{align*}
where $w'$ and $j'$ satisfy $S_w T_j = T_{j'}S_{w'}$.
If $j'\neq j$ then $T_{j'}^* T_j = 0$, in which case $\< S_w T_j h, T_j h\> = 0$.
If $j' = j$, then
\begin{align*}
    \< S_w T_j h, T_j h\> = \<S_{w'}h,h\> = 0,
\end{align*}
since $h$ is wandering for $S$ and $|w'|=|w|\geq 1$.
Hence $T_jh$ is wandering for $S$, and so $T_j h \in P^{\perp} H$.
It follows that $T_jP^{\perp} H \subseteq P^\perp H$, and hence $PH$ is $T^*$-invariant.
\end{proof}

\begin{proposition}\label{prop: Hu reduce}
Let $S=[S_1,\ldots,S_m]$ and $T = [T_1,\ldots, T_n]$ be $\theta$-commuting row-isometries on $H$.
Let $H = H_u \oplus H_s$ be the Wold decomposition for $S$.
If the unitary part of $S$ is singular, then $H_u$ reduces $T$
\end{proposition}

\begin{proof}
When $m = 1$, the result follows from \cite[Theorem~2.1]{Mla1972}, see \cite[Remark~2]{Slo1980}.
Otherwise, we have $H_u = PH$ where $P$ is the structure projection for $S$. The result follows from Lemma~\ref{lem: Hu inv} and Lemma~\ref{lem: PH *-invariant}.
\end{proof}

We now give a row-isometry analogue of \cite[Theorem~4]{Slo1980}.

\begin{theorem}\label{thm: singular parts}
Let $S=[S_1,\ldots,S_m]$ and $T=[T_1,\ldots,T_n]$ be $\theta$-commuting be a row-isometries on a Hilbert space $H$.
Further, suppose that the unitary parts of $S$ and $T$ are singular.
Then $S$ and $T$ have a \sloc-Wold decomposition.
\end{theorem}

\begin{proof}
The result follows immediately from Proposition~\ref{prop: sloc exists} and Proposition~\ref{prop: Hu reduce}.
\end{proof}

The following lemma generalises \cite[Lemma~2]{Slo1980} to row-isometries.
It is notable that the conditions are less restrictive for the row-isometry case than they are in single isometry case dealt with in \cite{Slo1980}.

\begin{lemma}\label{lem: unitary and shift 1}
Let $S$ be an $m$-shift of finite multiplicity on a Hilbert space $H$.
Let $T = [T_1,\ldots,T_n]$ be a row-unitary on $H$ which $\theta$-commutes with $S$.
If
\begin{enumerate}
    \item $n\geq 2$, or
    \item $n = 1$ and $T$ has empty point spectrum,
\end{enumerate}
then $H = \{0\}$.
\end{lemma}

\begin{proof}
Let $L = \bigcap_{i=1}^m \ker S_i^*$.
By assumption, $L$ is finite-dimensional.
Since $T$ and $S$ $\theta$-commute, it is clear that $L$ is $T^*$-invariant.
As $T$ is a row-unitary, if $h \in L$  and $1\leq i \leq m$ we have that
$$ S_i^* T_j h = \sum_{k=1}^n T_kT_k^* S_i^* T_j h = \sum_{\theta(i,k) = (i_k, j)}T_k S_{i_k}^* h = 0, $$
and so $L$ is $T$-reducing.

If $n\geq 2$, then $T_1|_L,\ldots, T_n|_L$ are isometries with pairwise orthogonal finite-dimensional ranges.
If $n=1$, then $T|_L$ is a unitary on a finite-dimensional space and so has an eigenvalue.
In either case we see that we must have $L = \{0\}$ and hence $H = \{0\}$.
\end{proof}

We end with the following generalisation of \cite[Theorem~5]{Slo1980}.

\begin{theorem}\label{thm: singular parts 2}
Let $S=[S_1,\ldots,S_m]$ and $T=[T_1,\ldots,T_n]$ be $\theta$-commuting be a row-isometries on a Hilbert space $H$.
Assume the unitary part of $S$ is singular, and the shift part of $S$ has finite multiplicity, then $S$ and $T$ have a \sloc-Wold decomposition if
\begin{enumerate}
    \item $n\geq 2$; or
    \item $n = 1$ and $\theta$ is the identity permutation.
\end{enumerate}
\end{theorem}

\begin{proof}
Let $H = H_u^S \oplus H_s^S$. 
As $S$ has only singular unitary part, $H_u^S$ reduces $T$ by Proposition~\ref{prop: Hu reduce}.
Let $H_s^S = K_u^T \oplus K_s^T$ be the Wold decomposition of the restriction of $T$ to $H_s^S$.
Lemma~\ref{lem: Hu inv} says that $K_u^T$ is $S$-invariant.
As $S$ is an $m$-shift of finite multiplicity on $H_s^S$, the restriction of $S$ to $K_u^T$ is an $m$-shift of finite multiplicity.
When $m=1$ this is \cite[Lemma~4]{Hal1961}; when $m\geq 2$  it follows from \cite[Theorem~3.1]{Pop1989b} and \cite[Theorem~3.2]{Pop1989b}.

When $n \geq 2$ it follows from Lemma~\ref{lem: unitary and shift 1} that $K_u^T = \{0\}$ and hence $S$ and $T$ have a \sloc-Wold decomposition by Proposition~\ref{prop: sloc exists}.
When $n=1$ and $T$ is an isometry commuting with each $S_i$ the proof follows as in \cite[Theorem~4]{Slo1980}.
\end{proof}

%\section{Odometer Semigroups}

\bibliographystyle{plain}
\bibliography{refs}

\begin{thebibliography}{10}

\bibitem{Dav2006}
Kenneth~R. Davidson.
\newblock {$\mathcal{B}(\mathcal{H})$} is a free semigroup algebra.
\newblock {\em Proc. Amer. Math. Soc.}, 134(6):1753--1757, 2006.

\bibitem{DKP2001}
Kenneth~R. Davidson, Elias Katsoulis, and David~R. Pitts.
\newblock The structure of free semigroup algebras.
\newblock {\em J. Reine Angew. Math.}, 533:99--125, 2001.

\bibitem{DLP2005}
Kenneth~R. Davidson, Jiankui Li, and David~R. Pitts.
\newblock Absolutely continuous representations and a {K}aplansky density
  theorem for free semigroup algebras.
\newblock {\em J. Funct. Anal.}, 224(1):160--191, 2005.

\bibitem{DPY2010}
Kenneth~R. Davidson, Stephen~C. Power, and Dilian Yang.
\newblock Dilation theory for rank 2 graph algebras.
\newblock {\em J. Operator Theory}, 63(2):245--270, 2010.

\bibitem{DavYan2009}
Kenneth~R. Davidson and Dilian Yang.
\newblock Representations of higher rank graph algebras.
\newblock {\em New York J. Math.}, 15:169--198, 2009.

\bibitem{FulYan2015}
Adam~H. Fuller and Dilian Yang.
\newblock Nonself-adjoint 2-graph algebras.
\newblock {\em Trans. Amer. Math. Soc.}, 367(9):6199--6224, 2015.

\bibitem{Ful2011}
Adam~Hanley Fuller.
\newblock Finitely correlated representations of product systems of
  {$C^*$}-correspondences over {$\mathbb{N}^k$}.
\newblock {\em J. Funct. Anal.}, 260(2):574--611, 2011.

\bibitem{GasSuc1989}
Dumitru Ga\c{s}par and Nicolae Suciu.
\newblock Wold decompositions for commutative families of isometries.
\newblock {\em An. Univ. Timi\c{s}oara Ser. \c{S}tiin\c{t}. Mat.},
  27(2):31--38, 1989.

\bibitem{Hal1961}
Paul~R. Halmos.
\newblock Shifts on {H}ilbert spaces.
\newblock {\em J. Reine Angew. Math.}, 208:102--112, 1961.

\bibitem{Ken2013}
Matthew Kennedy.
\newblock The structure of an isometric tuple.
\newblock {\em Proc. Lond. Math. Soc. (3)}, 106(5):1157--1177, 2013.

\bibitem{KumPas2000}
Alex Kumjian and David Pask.
\newblock Higher rank graph {$C^\ast$}-algebras.
\newblock {\em New York J. Math.}, 6:1--20, 2000.

\bibitem{Mla1972}
W.~Mlak.
\newblock Intertwining operators.
\newblock {\em Studia Math.}, 43:219--233, 1972.

\bibitem{MuhSol1999}
Paul~S. Muhly and Baruch Solel.
\newblock Tensor algebras, induced representations, and the {W}old
  decomposition.
\newblock {\em Canad. J. Math.}, 51(4):850--880, 1999.

\bibitem{Pop1989}
Gelu Popescu.
\newblock Isometric dilations for infinite sequences of noncommuting operators.
\newblock {\em Trans. Amer. Math. Soc.}, 316(2):523--536, 1989.

\bibitem{Pop1989b}
Gelu Popescu.
\newblock Multi-analytic operators and some factorization theorems.
\newblock {\em Indiana Univ. Math. J.}, 38(3):693--710, 1989.

\bibitem{Pop1996}
Gelu Popescu.
\newblock Non-commutative disc algebras and their representations.
\newblock {\em Proc. Amer. Math. Soc.}, 124(7):2137--2148, 1996.

\bibitem{Rea2005}
Charles~J. Read.
\newblock A large weak operator closure for the algebra generated by two
  isometries.
\newblock {\em J. Operator Theory}, 54(2):305--316, 2005.

\bibitem{SkaZac2008}
Adam Skalski and Joachim Zacharias.
\newblock Wold decomposition for representations of product systems of
  {$C^*$}-correspondences.
\newblock {\em Internat. J. Math.}, 19(4):455--479, 2008.

\bibitem{Slo1980}
Marek S{\l}oci\'{n}ski.
\newblock On the {W}old-type decomposition of a pair of commuting isometries.
\newblock {\em Ann. Polon. Math.}, 37(3):255--262, 1980.

\bibitem{Suc1968}
I.~Suciu.
\newblock On the semi-groups of isometries.
\newblock {\em Studia Math.}, 30:101--110, 1968.

\end{thebibliography}

\end{document}